\documentclass[12pt]{amsart}

\usepackage{amsmath}
\usepackage{amscd}
\usepackage{amsthm}
\usepackage{amsfonts}
\usepackage{graphicx}
\usepackage[all]{xy}

\newcommand{\cB}{\mathcal{B}}
\newcommand{\cC}{\mathcal{C}}

\newcommand{\cM}{\mathcal{M}}
\newcommand{\cN}{\mathcal{N}}

\newcommand{\cP}{\mathcal{P}}

\newcommand{\cU}{\mathcal{U}}
\newcommand{\cV}{\mathcal{V}}

\newcommand{\fR}{\mathbb{R}}

\newcommand{\fZ}{\mathbb{Z}}

\newcommand{\fN}{\mathbb{N}}

\swapnumbers
\theoremstyle{plain}
\newtheorem{stam}{STAM}[section]
\newtheorem{lem}[stam]{Lemma}
\newtheorem{thm}[stam]{Theorem}
\newtheorem{prop}[stam]{Proposition}
\newtheorem{cor}[stam]{Corollary}
\newtheorem*{lem*}{Lemma}
\newtheorem*{thm*}{Theorem}
\newtheorem*{prop*}{Proposition}
\newtheorem*{cor*}{Corollary}

\theoremstyle{definition}
\newtheorem{definition}[stam]{Definition}

\newtheorem*{definition*}{Definition}

\theoremstyle{remark}

\numberwithin{equation}{section}

\newenvironment{alist}{\begin{list}{(\alph{ctr})}{\usecounter{ctr}}}{\end{list}}

\begin{document}
\newcounter{ctr}
\title{Measure Theoretical Entropy of Covers}
\author{Uri Shapira}
\thanks{* Part of the author's MS.c thesis at the Hebrew University of Jerusalem.\\ Email: ushapira@gmail.com}
\begin{abstract}
In this paper we introduce three notions of measure theoretical
entropy of a measurable cover $\cU$ in a measure theoretical
dynamical system. Two of them were already introduced in [R] and
the new one is defined only in the ergodic case. We then prove
that these three notions coincide, thus answering a question posed
in [R] and recover a variational inequality (proved in [GW]) and a
proof of the classical variational principle based on a comparison
between the entropies of covers and partitions.
\end{abstract}

\maketitle

\section{Introduction}
In this paper a measure theoretical dynamical system (m.t.d.s) is a four
tuple $(X,\cB,\mu,T)$, where $(X,\cB)$ is a standard space (i.e isomorphic
to $[0,1]$ with the Borel $\sigma -algebra$ ,$\mu$ is a probability
measure on $(X,\cB)$ and $T$ is an invertible measure preserving map from
$X$ to itself.

A topological dynamical system (t.d.s) is a pair $(X,T)$, where $X$ is a
compact metric space and $T$ is a homeomorphism from $X$ to itself.

In [R] the author introduced two notions of measure theoretical entropy of
a
cover, both generalizing the definition of measure theoretical entropy of a
partition and influenced by [BGH]. Namely,
\begin{enumerate}

\item $h^+_\mu(\cU)=inf_{\alpha \succeq \cU}h_\mu(\alpha)$
\item $h^-_\mu(\cU)=lim\frac{1}{n}inf_{\alpha \succeq
\cU_0^{n-1}}H_\mu(\alpha)$

\end{enumerate}
It was shown there among other things that $h^-_\mu(\cU)\le h^+_\mu(\cU)$ and that in the
topological case (i.e a t.d.s and an open cover), one can
always find an invariant measure $\mu$ such that
$h^-_\mu(\cU)=h_{top}(\cU)$. This generalizes the result from [BGH] asserting that
in the topological case one can always find an invariant measure $\mu$ such that
$h_\mu^+(\cU)\ge h_{top}(\cU)$ \\

The question whether $h^-_\mu(\cU)= h^+_\mu(\cU)$ arose. In [HMRY] the
authors continued the research on these concepts and proved, among other
results, with aid of the Jewett-Krieger theorem, that if there exists a
t.d.s, an invariant measure $\mu$ and an open cover $\cU$ such that
$h^-_\mu(\cU)< h^+_\mu(\cU)$ then one can find such a situation in a
uniquely ergodic t.d.s.
\\
Recently, B.Weiss and E.Glasner [GW] showed that if $(X,T)$ is a t.d.s
and $\cU$ is any cover, then for any invariant measure $\mu$
$h^+_\mu(\cU)\le h_{top}(\cU)$
and so combining these results one concludes that for a t.d.s and an open cover we have that
$h^-_\mu(\cU)= h^+_\mu(\cU)$.
\\

The measure theoretical entropy of a partition $\alpha$ in an
ergodic m.t.d.s can be defined as:
$lim\frac{1}{n}log\cN(\alpha_0^{n-1},\epsilon)$, where
$0<\epsilon<1$ and
$\cN(\alpha_0^{n-1},\epsilon)$ is the minimum number of atoms of
$\alpha_0^{n-1}$ needed to cover $X$ up to a set of measure, less than
$\epsilon$. (See [Ru]).
\\
In this paper we follow this line and in section 4 define a notion
of measure theoretical entropy for a cover $\cU$ of an ergodic
m.t.d.s as $h_\mu^e(\cU) = lim \frac{1}{n}log\cN(\cU_0^{n-1} ,
\epsilon)$ (where $0<\epsilon<1$). We prove (Theorem 4.2) the
existence of the limit and its Independence of $\epsilon$, in a
different way from [Ru] using Strong Rohlin Towers. This can serve
as an alternative proof of the fact that the above definition of
measure theoretical entropy of a partition in an ergodic m.t.d.s is
well defined.
\\
We show in a direct way that in the ergodic case the three
notions: $h_\mu^-(\cU)$, $h_\mu^+(\cU), h_\mu^e(\cU)$, coincide
(Theorems 4.4, 4.5), and from the ergodic decomposition for
$h_\mu^-(\cU), h_\mu^+(\cU)$, proved in [HMRY], we deduce that
$h_\mu^-(\cU)= h_\mu^+(\cU)$ in the general case (Corollary 5.2),
and so, we can denote this number by $h_\mu(\cU,T)$ or
$h_\mu(\cU)$.
\\
We also get an immediate proof of a slight generalization of the
inequality $h_\mu(\cU)\le h_{top}(\cU)$, mentioned earlier, from
[GW], to the non topological case (Theorem 6.1).
\\

$\textbf{Acknowledgements}:$ This paper was written as an M.Sc thesis
at the Hebrew University of Jerusalem under the supervision of prof'
Benjamin Weiss. I would like to thank prof' Weiss, for introducing
me to the subject and for sharing with me his and Eli Glasner's
valuable ideas.

\section{Preliminaries}

Recall that in the following a measure theoretical dynamical system,
(m.t.d.s), is a four tuple
$(X,\cB,\mu,T)$,
where
$(X,\cB)$
is a standard space, $\mu$ is a probability measure on $(X,\cB)$ and $T$
is an invertible measure preserving transformation of $X$.
\\

\begin{definition}
\hspace{1cm}
\begin{itemize}
\item A cover of $X$ is a finite collection of measurable sets that cover
$X$.
\item The collection of covers of $X$ will be denoted by $\cC_X$
\item A partition of $X$ is a cover of $X$ whose elements are mutually
disjoint.
\item The collection of partitions of $X$ will be denoted by $\cP_X$.
\\
Usually we denote covers by $\cU, \cV$ and partitions by $\alpha, \beta,
\gamma$ etc.
\item We say that a cover $\cU$ is finer than $\cV$ ($\cU\succeq\cV$) if
any element of $\cU$ is contained in an element of $\cV$.
\item For any $\cU \in \cC_X$ and $k \in \fZ$ we denote by $T^k(\cU)$ the
cover whose elements are the sets of the form $T^k(U)$ where $U \in \cU$.
\item We define the join, $\cU \lor \cV$, of two covers $\cU, \cV$, to be
the cover whose elements are sets of the form $U \cap V$ where $U\in\cU$
and $V\in\cV$.
\item When the transformation $T$ is understood we denote, for $l>k$, the
cover $T^{-k}(\cU) \lor T^{-(k+1)}(\cU) \dots\lor T^{-l}(\cU)$, by
$\cU_k^l$.

\end{itemize}
\end{definition}

\begin{definition}
For $0<\delta<1$ define $H(\delta)=-\delta log\delta -
(1-\delta)log(1-\delta)$.
Note that $lim_{\delta \to 0}H(\delta)=0.$
\end{definition}

In the sequel, we will prove some combinatorial lemmas and often we
will encounter the expression $\sum_{j \le \delta K}\binom{K}{j}$. We
shall make use of the next elementary lemma:

\begin{lem}
(lemma 1.5.4 in [Sh1]): If $\delta < \frac{1}{2}$ then $\sum_{j \le \delta
K}\binom{K}{j} \le 2^{H(\delta)}$.
\end{lem}

\begin{definition}
A m.t.d.s $(X,\cB,\mu,T)$ is said to be aperiodic, if for every $n \in
\fN$, $\mu(\{x|T^nx=x\})=0$.
\end{definition}

An ergodic system which is not aperiodic is easily seen to be a cyclic
permutation on a finite number of atoms.\\
One of our main tools in practice, will be the Strong Rohlin Lemma ([Sh2]
p.15):
\begin{lem}
Let $(X,\cB,\mu,T)$ be an ergodic, aperiodic system and let $\alpha \in
\cP_X$. Then for any $\delta > 0$ and $n \in \fN$, one can find a set $B
\in \cB$, such that $B,TB \dots ,T^{n-1}B$ are mutually disjoint,
$\mu(\bigcup_0^{n-1}T^iB) > 1- \delta$ and the distribution of $\alpha$ is
the same as the distribution of the partition $\alpha|_B$ that $\alpha$
induces on $B$.
\end{lem}
The data $(n,\delta,B,\alpha)$ will be called, a strong Rohlin tower of
height $n$ and error $\delta$ with respect to $\alpha$ and with $B$ as a
base. \\

\section{Measure theoretical entropy of covers}
Let $(X,\cB,\mu,T)$ be a m.t.d.s. The definitions and proofs in this
section were introduced in [R].
\begin{definition}
for $\cU \in \cC_X$ we define the entropy of $\cU$ as:\\
$H_\mu(\cU)= inf_{\alpha \succeq \cU} H_\mu(\alpha)$.
\end{definition}

\begin{prop}
\hspace{1cm}
\begin{enumerate}
\item If $\cU , \cV \in \cC_X$ then $H_\mu(\cU \lor \cV) \le H_\mu(\cU)+
H_\mu(\cV).$
\item For every $\cU \in \cC_X$ $H_\mu(T^{-1}\cU)= H_\mu(\cU)$
\end{enumerate}
\end{prop}

\begin{cor}
If $\cU \in \cC_X$ then the sequence $H_\mu(\cU_0^{n-1})$ is sub-additive.
\end{cor}

\begin{cor}
If $\cU \in \cC_X$ then the sequence $\frac{1}{n}H_\mu(\cU_0^{n-1})$
converges to $inf_n\frac{1}{n}H_\mu(\cU_0^{n-1})$.
\end{cor}

Two ways of generalizing the definition of measure theoretical entropy
of a partition to a cover are:
\begin{definition}
If $\cU \in \cC_X$, define
\begin{enumerate}
\item $h_\mu^-(\cU,T)=lim\frac{1}{n}H_\mu(\cU_0^{n-1})$.
\item $h_\mu^+(\cU,T)= \inf_{\alpha \succeq \cU}h_\mu(\alpha,T)$.
\end{enumerate}
When $T$ is understood we usually omit it and write $h_\mu^-(\cU)$,
$h_\mu^+(\cU)$.
\end{definition}
We shall see later that in fact $h_\mu^-(\cU)=h_\mu^+(\cU)$.
\begin{prop}
\hspace{1cm}
\begin{enumerate}
\item $h_\mu^-(\cU) \le h_\mu^+(\cU)$.
\item for any $m \in \fN \quad h_\mu^-(\cU,T)=\frac{1}{m}
h_\mu^-(\cU_0^{m-1},T^m)$
\item $h_\mu^-(\cU,T)=lim_n\frac{1}{n} h_\mu^+(\cU_0^{n-1},T^n)$
\end{enumerate}
\end{prop}

\section{The ergodic case}

Throughout this section, $(X,\cB,\mu,T)$, is an ergodic m.t.d.s.\\
For $\cU \in \cC_X$, we denote by $\cN(\cU,\epsilon,\mu)$, the
minimum number of elements of $\cU$, needed to cover all of $X$, up to a
set of measure, less than $\epsilon$. When $\mu$ is understood we write
$\cN(\cU,\epsilon)$.

By a strait forward calculation one deduces from [Sh1] p.51 the
following:
\begin{thm}
If $(X,\cB,\mu,T)$ is an ergodic m.t.d.s and $\alpha \in \cP_X$, then for
any $0< \epsilon <1$,
$h_\mu(\alpha,T)=lim\frac{1}{n}log
\cN(\alpha_0^{n-1},\epsilon)$.
\end{thm}

In view of this result, a natural way to generalize the definition
of measure theoretical entropy of a partition to covers will be the
following:
\[h_\mu(\cU,T)=lim\frac{1}{n}log\cN(\cU_0^{n-1},\epsilon).\]
Where $0< \epsilon <1$. In order to do so we have to show that the above
limit exists and is independent of $\epsilon$.

\begin{thm}
For any $0< \epsilon <1$, the sequence
$\frac{1}{n}log\cN(\cU_0^{n-1},\epsilon)$
converges and the limit is independent of $\epsilon$.
\end{thm}
In order to prove this theorem we shall need a combinatorial lemma.
Let us first introduce some terminology (in first reading the reader
may skip the following discussion and turn to the discussion held
after the proof of Lemma 4.3):
\begin{itemize}
\item We say that two intervals in $\fN$, $I,J$ are separated if there is
$n \in \fN$ such that for any $i \in I, j \in J$ we have
$i<n<j$ or $j<n<i$.
\item We say that a collection $\{I_i\}_{i \in A}$ of intervals in
$\fN$ is a separated collection if any two of its elements are separated.
\item We say that a collection $\{I_i\}_{i \in A}$ of subintervals of an interval
$[1,K]$ is a $(\lambda,\epsilon)$ separated  cover of $[1,K]$ (for $0<
\lambda <1$, $0<\epsilon$), if it is separated and
\[|\frac{|\cup{I_i}|}{K}-\lambda|< \epsilon .\]
\item Given a vector $\vec{\lambda}=(\lambda_1\dots\lambda_l)$, we
denote $$\nu_r(\vec{\lambda})=\prod_{j=r}^l(1-\lambda_j)$$ or just
$\nu_r$ when $\vec{\lambda}$ is understood. For $r>l$ we set
$\nu_r=1$. Note that for $j<l$ we have:
$$\sum_{r=j+1}^l\lambda_r\nu_{r+1}=1-\nu_j.$$
\end{itemize}

In the following combinatorial lemma, we will be given $l$ separated
collections $\{I_i^j\}_{i \in A_j}$, $ j=1 \dots l$ of subintervals
of a very long interval $[1,K]$. The knowledge about these
collections is that the members of the $j$'th collection all have
the same length, $N_j$, $N_1<<N_2 \dots <<N_l$ and every collection
is very "equally distributed"
in $[1,K]$ in some sense.\\
We would like to extract, from these collections, a separated collection
that will cover as much as we can, from $[1,K]$.\\
Let us denote by $\lambda_j$, the percentage of $[1,K]$, that is
covered by the $j$'th collection and by $\vec{\lambda}$, the
corresponding vector. Then, $\lambda_l=1-\nu_l$ percent of $[1,K]$
is covered by $\{I_i^l\}$. The complement is of size $K\nu_l$ and we
could cover $\lambda_{l-1}$ percent of it with the
$\{I_i^{l-1}\}$'s. By now we covered $K(1-\nu_{l-1})$ and we could
cover $\lambda_{l-2}$ percent of the complement by the
$\{I_i^{l-2}\}$'s. So by now we covered $K(1-\nu_{l-2})$ of $[1,K]$.
We go on this way and extract a separated collection that covers
$1-\nu_1$ percent of $[1,K]$. Let us now make these ideas precise.

\begin{lem}
For any $l>0$, there exists a positive function
$\varphi= \varphi(N_1 \dots N_l,\eta_1 \dots \eta_l,\epsilon)$
(where $N_1<N_2 \dots <N_l \in \fN$, $\eta_i, \epsilon>0$)
such that
\[\limsup_{\epsilon \rightarrow 0}
\limsup_{N_1 \rightarrow \infty}
\limsup_{\eta_1 \rightarrow 0}
\dots
\limsup_{N_l \rightarrow \infty}
\limsup_{\eta_l \rightarrow 0}
\varphi(N_i,\eta_i,\epsilon)=0. \qquad (*)\]
and such that if $0< \lambda_j <1$ $j=1 \dots l$ and
$\{I_i^j\}_{i \in A_j}$ are separated collections of subintervals of
$[1,K]$ that satisfy:
\begin{alist}
\item For every $1 \le j \le l$ $|I_i^j|=N_j$.
\item For every $1 \le j \le l$ $\{I_i^j\}$ is a
$(\lambda_j,\epsilon)$-separated cover of $[1,K]$.
\item For every $0 \le j < r \le l$, the number of subintervals, $J$, of
$[1,K]$, of length $N_r$, which are not $(\lambda_j,\epsilon)$-separately
covered by $\{I_i^j \subset J\}$ is less than $\eta_rK$.
\end{alist}
then there are sets $\tilde{A_j} \subset A_j$ $j=1 \dots l$, such
that $\{\{I_i^j\}_{i \in \tilde{A_j}}\}_{j=1}^l$ is a separated
collection and $[1,K]$ is
$((1-\nu_1(\vec{\lambda})),\varphi(N_i,\eta_i,\epsilon))$-separately
covered by $\{\{I_i^j\}_{i \in \tilde{A_j}}\}_{j=1}^l.$
\end{lem}

\begin{proof}
We will build the $\tilde{A_j}$'s by recursion, starting with $j=l$.
Define $\tilde{A_l}=A_l$. Then from (b) we have that
$|\frac{N_l|\tilde{A_l}|}{K}- \lambda_l|<\epsilon$. So if we will
define $f_l(N_i,\eta_i,\epsilon)=\epsilon$, then $f_l$ satisfies
$(*)$ and $[1,K]$ is
$(\lambda_l\nu_{l+1},f_l(N_i,\eta_i,\epsilon))$-separately covered
by $\{I_i^l\}_{i \in \tilde{A_l}}$. Now, suppose we have defined
$\tilde{A_l} \dots \tilde{A}_{j+1}$ and positive functions $f_l
\dots f_{j+1}$, that satisfy $(*)$, such that $\{\{I_i^r\}_{i \in
\tilde{A_r}}\}_{r=j+1}^l$, is a separated collection and for every
$j+1 \le r \le l$, $[1,K]$ is
$(\lambda_r\nu_{r+1},f_r(N_i,\eta_i,\epsilon))$-separately covered
by $\{I_i^r\}_{i \in \tilde{A_r}}$. Define now,
\[\tilde{A_j}=\{i \in A_j |\;I_i^j\; is\; separated\; from\; \{I_s^r\}_{s
\in \tilde{A_r}}, r=j+1 \dots l\}.\]
We want to estimate the size of $\tilde{A_j}$.\\
Estimation from below: Choose $j+1 \le r \le l$ and divide the members of
$\{I_i^r\}_{i \in \tilde{A_r}}$ to good ones and bad ones according to
(c), i.e, $I_s^r$ is good if it is $(\lambda_j,\epsilon)$-separately
covered by $\{I_i^j \subset I_s^r\}$. We have at most $\eta_rK$,
$I_i^r$'s,
which are bad and at most $|\tilde{A_r}|$, $I_i^r$'s, which are good.
Every
bad $I_i^r$ rules out at most $\frac{N_r}{N_j}+2$ $i$'s in $A_j$ from
being
in $\tilde{A_j}$. Every good $I_i^r$ rules out at most
$\frac{N_r}{N_j}(\lambda_j+\epsilon)+2$, $i$'s in $A_j$ from being in
$\tilde{A_j}$. In total, the maximum number of $i$'s in $A_j$ that are not
in $\tilde{A_j}$ is at most:
\[ \sum_{r=j+1}^l|\tilde{A_r}|(\frac{N_r}{N_j}(\lambda_j+\epsilon)+2)+
\eta_rK(\frac{N_r}{N_j}+2)= (**)\] Note that because $[1,K]$ is
$(\lambda_r\nu_{r+1},f_r)$-separately covered by $\{I_i^r\}_{i \in
\tilde{A_r}}$, we must have
\[ |\tilde{A_r}| \le
\frac{K}{N_r}(\lambda_r\nu_{r+1}+f_r).\] Using this we get:
\[(**) \le \sum_{r=j+1}^l\frac{K}{N_r}(\lambda_r\nu_{r+1}+f_r)(\frac{N_r}{N_j}(\lambda_j+\epsilon)+2)+
\eta_rK(\frac{N_r}{N_j}+2)\]
\[=\sum_{r=j+1}^l\frac{K}{N_j} \lambda_r
\nu_{r+1}(\lambda_j+\epsilon)+ \frac{K}{N_j}(\lambda_j+\epsilon)f_r+
\frac{2K}{N_r}(\lambda_r \nu_{r+1}+f_r)+
\frac{K}{N_j}\eta_rN_r+2\eta_rK
\]
\[=\frac{K}{N_j}\lambda_j
(\sum_{r=j+1}^l\lambda_r\nu_{r+1})\]
\[
+\frac{K}{N_j}\sum_{r=j+1}^l\{\epsilon\lambda_r\nu_{r+1}
+(\lambda_j+\epsilon)f_r+2\frac{N_j}{N_r}
(\lambda_r\nu_{r+1}+f_r)+\eta_r(N_r+2N_j)\} =(\aleph)\] as mentioned
earlier $\sum_{j+1}^l\lambda_r\nu_{r+1}=1-\nu_j$ so we have that:
\[|\tilde{A_j}| \ge |A_j|- (\aleph) \ge
\frac{K}{N_j}(\lambda_j-\epsilon)-(\aleph)\]
\[=\frac{K}{N_j}\bigg\{\lambda_j\nu_j-
\Big\{\epsilon
+\sum_{r=j+1}^l\{\epsilon\lambda_r\nu_{r+1}+(\lambda_j+\epsilon)f_r+2\frac{N_j}{N_r}
(\lambda_r\nu_{r+1}+f_r)+\eta_r(N_r+2N_j)\}\Big\}\bigg\}\] note that

 \[|(\epsilon +\sum_{r=j+1}^l\{\epsilon\lambda_r\nu_{r+1}+(\lambda_j+\epsilon)f_r+2\frac{N_j}{N_r}
(\lambda_r\nu_{r+1}+f_r)+\eta_r(N_r+2N_j)\}|\]
\[\le \epsilon+\sum_{r=j+1}^l\{\epsilon+(1+\epsilon)f_r+2\frac{N_j}{N_r}
(1+f_r)+\eta_r(N_r+2N_j)\}\]

so if we will denote the last expression by
$\tilde{f_j}(N_i,\eta_i,\epsilon)$, then we see that $\tilde{f_j}$
satisfies $(*)$ and $ |\tilde{A_j}|\ge \frac{K}{N_j}(\lambda_j
\nu_{j+1}- \tilde{f_j})$.

Estimation from above: For every $j+1\le r\le l$, we have that
$|\tilde{A_r}|\ge\frac{K}{N_r}(\lambda_r\nu_{r+1}-f_r)$ and the
number of bad $I_i^r$'s is at most $\eta_rK$, so we must have at
least $\frac{K}{N_r}(\lambda_r\nu_{r+1}-f_r)-\eta_rK$ good
$I_i^r$'s. Every good $I_i^r$, rules out at least $\frac{N_r}{N_j}
(\lambda_j-\epsilon)$ $i$'s in $A_j$ from being in $\tilde{A_j}$.
So the number of $i$'s in $A_j$ that are not in $\tilde{A_j}$ is
at least:
\[
\sum_{r=j+1}^l\frac{N_r}{N_j}(\lambda_j-\epsilon)\{\frac{K}{N_r}(\lambda_r\nu_{r+1}-f_r)
-\eta_rK\}\] and so
\[|\tilde{A_j}|\le |A_j|-
\sum_{r=j+1}^l\frac{N_r}{N_j}(\lambda_j-\epsilon)\{\frac{K}{N_r}(\lambda_r\nu_{r+1}-f_r)
-\eta_rK\}\]
\[\le
\frac{K}{N_j}(\lambda_j+\epsilon)-
\sum_{r=j+1}^l\Big\{\frac{K}{N_j}\Big(\lambda_j
(\lambda_r\nu_{r+1}-f_r)-\epsilon (\lambda_r\nu_{r+1}-f_r)\Big)-
\frac{K}{N_j}\eta_rN_r(\lambda_j-\epsilon)\Big\}\]
\[=\frac{K}{N_j}\Big\{
\lambda_j\Big(1-\sum_{r=j+1}^l\lambda_r\nu_{r+1}\Big)
+\epsilon+\sum_{r=j+1}^l\Big(\lambda_jf_r+\epsilon
(\lambda_r\nu_{r+1}-f_r)
+\eta_rN_r(\lambda_j-\epsilon)\Big)\Big\}\]
\[\le \frac{K}{N_j}\Big\{
\lambda_j\nu_{j+1}+\epsilon+
\sum_{r=j+1}^l\Big(f_r+\epsilon(1+f_r)+\eta_rN_r(1+\epsilon)\Big)\Big\}\]
so if we will denote
\[\hat{f_j}(N_i,\eta_i,\epsilon)=\epsilon+
\sum_{r=j+1}^l\Big(f_r+\epsilon(1+f_r)+\eta_rN_r(1+\epsilon)\Big)\Big\}\]
then $\hat{f_j}$ satisfies $(*)$ and
$|\tilde{A_j}|\le\frac{K}{N_j}\Big(\lambda_j\nu_{j+1}
+\hat{f_j}\Big)$. Define $f_j=max(\tilde{f_j},\hat{f_j})$ and then
we have that $f_j$ satisfies $(*)$ and
\[|\frac{|\tilde{A_j}|N_j}{K}-
\lambda_j\nu_{j+1}|\le f_j.\] We have defined $\tilde{A_j}\subset
A_j$ and a positive function $f_j$, that satisfies $(*)$, such
that $\{\{I_i^r\}_{i\in \tilde{A_r}}\}_{r=j}^l$ is a separated
collection and $[1,K]$ is $(\lambda_j\nu_{j+1},f_j)$-separately
covered by
$\{I_i^j\}_{i\in \tilde{A_j}}$.\\
We continue this way and define sets $\tilde{A_j}\subset A_j$ and
positive functions $f_j$, $j=1\dots l$, such that $\{\{I_i^j\}_
{i\in \tilde{A_j}}\}_{j=1}^l$, is a separated collection and
$[1,K]$ is $(\lambda_j\nu_{j+1},f_j)$-separately covered by
$\{I_i^j\}_{i\in \tilde{A_j}}$.\\
Note that this means:
\[K\Big(\sum_{j=1}^l\lambda_j\nu_{j+1}-\sum_{j=1}^lf_r\Big)
\le|\bigcup_{j=1}^l\bigcup_{i\in \tilde{A_j}}I_i^j| \le
K\Big(\sum_{j=1}^l\lambda_j\nu_{j+1}+\sum_{j=1}^lf_r\Big)\] and
so, if we will define $\varphi = \sum f_j$, then $\varphi$
satisfies $(*)$ and $\{\{I_i^j\}_ {i\in \tilde{A_j}}\}_{j=1}^l$,
is a $(1-\nu_1,\varphi)$- separated cover of $[1,K]$.
\end{proof}

Before turning to the proof of $theorem\; 4.2$, let us present some
terminology. In the following $\cU=\{U_1\dots U_M\}$, is a cover of $X$.
For any $\rho>0$, we can find a partition $\beta\succeq\cU$, such that
$\cN(\cU,\rho)=\cN(\beta,\rho)$. Namely, we choose a subset of $\cU$, of
$N=\cN(\cU,\rho)$ elements, that covers $X$ up to a set of measure
$<\rho$, $\{U_{i1}\dots U_{iN}\}$ and define
$C_1=U_{i1}$,
$C_j=U_{ij}\setminus\bigcup_{m=1}^{j-1}U_{im}$, $j=2\dots N$. The $C_j$'s
are disjoint, $C_j \subset U_{ij}$ and
$\bigcup_1^NC_j=\bigcup_{j=1}^NU_{ij}$.
Extend the collection $\{C_j\}_{j=1}^N$ to a partition, $\beta$, refining
$\cU$, in some way. Then, because $\beta\succeq\cU$, we have
$\cN(\beta,\rho)\ge N$ and from our construction, it follows that
$\cN(\beta,\rho)\le N$.
\begin{itemize}
\item We call such a partition, a $\rho$-good partition for $\cU$.
\end{itemize}
If $(X,\cB,\mu,T)$ is aperiodic and $N\in\fN,\;\rho,\delta>0$ are given,
then for a $\rho$-good partition $\beta$, for $\cU_0^{N-1}$, we can
construct a strong Rohlin tower with height $N+1$ and error $<\delta$.
Let $\tilde{B}$ denote the base of the tower and let $B\subset\tilde{B}$
be a union of $\cN(\beta,\rho)$ atoms of $\beta|_{\tilde{B}}$ that covers
$\tilde{B}$ up to a set of measure, less than $\rho\mu(\tilde{B})$.
\begin{itemize}
\item We call $(\beta,\tilde{B},B)$, a good base for
$(\cU,N,\rho,\delta)$.
\item For a set $J\subset\fN$, a $(\cU,J)$-name, is a function
$f:J\rightarrow \{1\dots M\} $.
\item $f$ is a name of $x\in X$, if $x\in \bigcap_{j\in J}T^{-j}U_{f(j)}.$
\item We denote the set of elements of $X$ with $f$ as a name by $S_f$.
\item A set of $(\cU,J)$-names, $\{f_i\}$, covers a set $C\in \cB$, if
$C\subset\bigcup_iS_{f_i}$.
\end{itemize}
In the sequel, we will want to estimate the number of elements of
$\cU_0^{N-1}$, needed to cover a set $C\in\cB$, i.e, we will want to
estimate the number of $(\cU,[0,N-1])$-names needed to cover $C$.
The usual way to do so is to find a collection of disjoint sets
$J_i\subset [0,N-1]$ $i=1\dots m$, that covers most of $[0,N-1]$, such
that
we can bound the number of $(\cU,J_i)$-names needed to cover $C$.
If we can cover $C$ by $R_i$, $(\cU,J_i)$-names, $\{f_m^i\}_{m=1}^{R_i}$,
then the set\\
 $\Gamma=\{f:[0,N-1]\rightarrow\{1\dots M\}|\quad f|_{J_i}
\in \{f_m^i\}_{m=1}^{R_i}\}$, of $(\cU,[0,N-1])$-names, covers $C$ and
contains $\prod R_i\cdot M^{N-\sum|J_i|}$ elements.\\
This situation occurs in our proofs in the following way: Let
$(\beta,\tilde{B},B)$, be a good base for $(\cU,N,\rho,\delta)$
and $K>>N$. Set $C$ to be the set of elements of $X$ that visits
$B$ at times $i_1<\dots <i_m$ between $0$ to $K-N$ (under the
action of $T$). Then we can cover $C$ by no more than
$\cN(\beta,\rho)$, $(\cU,[i_j,i_j+N-1])$-names. We can now turn to
the proof of $theorem\;4.2$.
\begin{proof}
($theorem\;4.2$): If $(X,\cB,\mu,T)$ is periodic, it follows from the
ergodicity, that the system is a cyclic permutation on a finite set of
atoms and for every $0<\epsilon <1$ we have
$lim\frac{1}{n}log\cN(\cU_0^{n-1},\epsilon)=0$.
We assume, then, that the system is aperiodic and thus we are able to use
the Strong Rohlin Lemma. Given $0<\rho_2<\rho_1<1$, we need to show that
the limits: $lim\frac{1}{n}log\cN(\cU_0^{n-1},\rho_i) \; i=1,2$, exist
and
are equal. Note that for every $n$, we have that $\cN(\cU_0^{n-1},\rho_1)
\le\cN(\cU_0^{n-1},\rho_2)$ and thus
$limsup\frac{1}{n}log\cN(\cU_0^{n-1},\rho_1)\le
liminf\frac{1}{n}log\cN(\cU_0^{n-1},\rho_2)$, so it's enough to prove that
$$limsup\frac{1}{n}log\cN(\cU_0^{n-1},\rho_2)\le
liminf\frac{1}{n}log\cN(\cU_0^{n-1},\rho_1).$$
Let $0<\epsilon_0<\frac{1}{2}$, be given and denote:\\
 $h_0=
liminf\frac{1}{n}log\cN(\cU_0^{n-1},\rho_1)$, $L=\{n\in\fN|\;|h_0-
\frac{1}{n}log\cN(\cU_0^{n-1},\rho_1)|<\epsilon_0\}$,\\
so $L$ contains arbitrarily large numbers. Choose $\ell\in\fN$, large
enough so that
\[\big(\frac{1}{2}(1+\rho_1)\big)^\ell logM<\epsilon_0,\qquad
\big(\frac{1}{2}(1+\rho_1)\big)^\ell+\epsilon_0<\frac{1}{2}\qquad
(*).\] \textbf{The towers construction}: Remember the function $\varphi$
from the combinatorial lemma (Lemma 4.3). It satisfies:
\[\limsup_{\epsilon\rightarrow 0}
\limsup_{N_1\rightarrow\infty}\limsup_{\eta_1\rightarrow 0}
\dots
\limsup_{N_\ell\rightarrow\infty}\limsup_{\eta_\ell\rightarrow 0}
\varphi(N_i,\eta_i,\epsilon)=0\]
so we can choose $\epsilon>0$, small enough, such that
\[\limsup_{N_1\rightarrow\infty}\limsup_{\eta_1\rightarrow 0}
\dots
\limsup_{N_\ell\rightarrow\infty}\limsup_{\eta_\ell\rightarrow 0}
\varphi(N_i,\eta_i,\epsilon)<\epsilon_0.\]
Choose a small enough $\delta>0$ (in a manner specified later). Choose
$N_1\in L$, large enough, such that
\[\limsup_{\eta_1\rightarrow 0}
\dots
\limsup_{N_\ell\rightarrow\infty}\limsup_{\eta_\ell\rightarrow 0}
\varphi(N_i,\eta_i,\epsilon)<\epsilon_0.\]
Find a good base $(\beta_1,\tilde{B_1},B_1)$, for
$(\cU,N_1,\rho_1,\delta)$.
Choose $\eta_1>0$, small enough, such that
\[\limsup_{N_2\rightarrow\infty}\limsup_{\eta_2\rightarrow 0}
\dots
\limsup_{N_\ell\rightarrow\infty}\limsup_{\eta_\ell\rightarrow 0}
\varphi(N_i,\eta_i,\epsilon)<\epsilon_0.\]
From the ergodicity, we can choose $N_2\in L$, large enough, such that
\begin{itemize}
\item $\limsup_{\eta_2\rightarrow 0}
\dots
\limsup_{N_\ell\rightarrow\infty}\limsup_{\eta_\ell\rightarrow 0}
\varphi(N_i,\eta_i,\epsilon)<\epsilon_0.$
\item $\mu\{x\;|\;|\frac{1}{N_2}\sum_{r=0}^{N_2-N_1}\chi_{B_1}
(T^rx)-\mu(B_1)|<\frac{\epsilon}{N_1}\}>1-\eta_1.$
\end{itemize}
Find a good base, $(\beta_2,\tilde{B_2},B_2)$, for
$(\cU,N_2,\rho_1,\delta).$
Choose $\eta_2>0$, small enough, such that
\[\limsup_{N_3\rightarrow\infty}\limsup_{\eta_3\rightarrow 0}
\dots
\limsup_{N_\ell\rightarrow\infty}\limsup_{\eta_\ell\rightarrow 0}
\varphi(N_i,\eta_i,\epsilon)<\epsilon_0.\]
Again, from the ergodicity, we can choose $N_3\in L$, such that
\begin{itemize}
\item $\limsup_{\eta_3\rightarrow 0}
\dots
\limsup_{N_\ell\rightarrow\infty}\limsup_{\eta_\ell\rightarrow 0}
\varphi(N_i,\eta_i,\epsilon)<\epsilon_0.$
\item $\mu\{x\;|\;|\frac{1}{N_3}\sum_{r=0}^{N_3-N_j}\chi_{B_j}
(T^rx)-\mu(B_j)|<\frac{\epsilon}{N_j}\;j=1,2\}>1-\eta_2.$
\end{itemize}
In this way we construct, inductively, $N_1<N_2\dots<N_\ell$ (all from
$L$), $\eta_1\dots\eta_\ell$ and good bases $(\beta_j,\tilde{B_j},B_j)$,
for $(\cU,N_j,\rho_1,\delta)$, such that $\varphi(N_i,\eta_i,\epsilon)<
\epsilon_0$
and if we denote
\[F_j=\{x\;|\;|\frac{1}{N_j}\sum_{r=0}^{N_j-N_i}\chi_{B_i}(T^rx)-\mu(B_i)
|<\frac{\epsilon}{N_i}\;i=1\dots j-1\}\]
then, $\mu(F_j)>1-\eta_j$.\\
Define
\[E_K=\{x\;|\;
\frac{1}{K}\sum_{r=0}^{K-N_j}\chi_{F_j}(T^rx)>1-\eta_j,\;
|\frac{1}{K}\sum_{r=0}^{K-N_j}\chi_{B_j}(T^rx)-\mu(B_j)|<\frac{\epsilon}{N_j}
\quad j=1\dots\ell\}.\]
From the ergodicity, we know that there is a $K_0$, such that, for any
$K>K_0$, we have $\mu(E_K)>\rho_2$. Fix $K>K_0$, we shall show that we can
cover $E_K$, by "few" $(\cU,[0,K-1])$-names. For a fixed $x\in E_K$ denote
\[A_j=\{0\le m\le K-N_j\;|T^mx\in B_j\}\]
and for every $i\in A_j$, let
$I_i^j=[i,i+N_j-1]$.We claim that the collections $\{I_i^j\}_{i\in A_j}$
$j=1\dots\ell$, satisfies conditions $(a),(b),(c)$ from the combinatorial
lemma ($lemma\;4.3$), with $\lambda_j=N_j\mu(B_j)$. To see this, note
first, that because the height of the $j$'th tower was $N_j+1$, we have
that each collection $\{I_i^j\}_{i\in A_j}$, is separated.\\
$(a)$ By definition $|I_i^j|=N_j$.\\
$(b)$ because $x\in E_k$, we know that
$|\frac{1}{K}\sum_{r=0}^{K-N_j}\chi_{B_j}(T^rx)-\mu(B_j)|<\frac{\epsilon}{N_j}$
and thus, $|\frac{N_j|A_j|}{K}-\lambda_j|<\epsilon$.
So the $\{I_i^j\}_{i\in A_j}$ forms a $(\lambda_j,\epsilon)$-separated
cover of $[0,K-1]$.\\
$(c)$ For $1<r\le\ell$, we know from the fact that $x\in E_K$, that
$\frac{1}{K}\sum_{s=0}^{K-N_r}\chi_{F_r}(T^sx)>1-\eta_r$ and thus we have
$\frac{1}{K}\sum_{s=0}^{K-N_r}\chi_{F_r^c}(T^sx)<\eta_r$. If we use the
definition of $F_r$, this becomes
\[\frac{1}{K}\#\{0\le s\le K-N_r\;|\;\exists\;1\le j\le r-1\;
|\frac{1}{N_r}\sum_{i=0}^{N_r-N_j}\chi_{B_j}(T^{i+s}x)-\mu(B_j)|\ge
\frac{\epsilon}{N_j}\}<\eta_r\]
or equivalently
\[\#\{0\le s\le K-N_r\;|\;\exists\;1\le j\le r-1\;|\frac{N_j}{N_r}
\#\{i\;|\; i+s\in A_j\}-\lambda_j|\ge\epsilon\}<\eta_rK\]
so if we choose $1\le j<r\le\ell$, we must have
\[\#\{J\subset[0,K-1]\;|\;|J|=N_r,\;|\frac{N_j}{N_r}\#
\{i\;|\;I_i^j\subset J\}-\lambda_j|\ge\epsilon\}<\eta_r K.\]
In words, the number of subintervals of $[0,K-1]$ of length $N_r$, $J$,
which are not $(\lambda_j,\epsilon)$-separately covered, by those
$I_i^j$ which are contained in $J$ is less than $\eta_r K$, as we wanted.\\
Using the combinatorial lemma, we can choose for every $x\in E_K$
a separated collection $\{\{I_i^j(x)\}_{i\in
\tilde{A_j}}\}_{j=1}^\ell$ that covers at least
$K\big(1-\nu_1(\vec{\lambda})-\epsilon_0\big)$ elements of
$[0,K-1]$. Because these collections are separated, there is a
$1-1$ correspondence between them and their complements. Hence,
the number of such covers is less than
\[\psi(K,\lambda_j,\epsilon_0)=
\sum_{j\le\big(\nu_1+\epsilon_0\big)K}\binom{K}{j}\qquad (**)\]
Fix such a collection $\{\{I_i^j\}_{i\in\tilde{A_j}}\}_{j=1}^\ell$
and set
\[C=\{x\in E_K\;|\;\{I_i^j(x)\}=\{I_i^j\}\;\}.\]
From the construction we see that for every $1\le j\le\ell$ we can
cover $B_j$ by no more than $2^{N_j(h_0+\epsilon_0)}$
$(\cU,[0,N_j-1])$-names, thus we can cover $C$ by no more than
$2^{N_j(h_0+\epsilon_0)}$ $(\cU,I_i^j)$-names. So the number of
$(\cU,[0,K-1])$-names, needed to cover $C$ is at most
\[\prod_{j=1}^\ell(2^{N_j(h_0+\epsilon_0)})^{|\tilde{A_j}|}\cdot
M^{K(\nu_1+\epsilon_0)}=2^{(\sum_jN_j|\tilde{A_j}|)(h_0
+\epsilon_0)}\cdot M^{K(\nu_1+\epsilon_0)}\]
\[\le 2^{K(h_0+\epsilon_0)}\cdot M^{K(\nu_1+\epsilon_0)}.\]
Finally we get from this and $(**)$ that
\[\cN(\cU_0^{K-1}, \rho_2)\le
\psi(K,\lambda_j,\epsilon_0)\cdot 2^{K(h_0+\epsilon_0)}\cdot
M^{K(\nu_1+\epsilon_0)}\] and so
\[\frac{1}{K}log\cN(\cU_0^{K-1}, \rho_2)\le
\frac{1}{K}log\psi(K,\lambda_j,\epsilon_0)+h_0+\epsilon_0+
\nu_1logM+\epsilon_0logM.\] If, in the construction of the towers,
we choose $\delta$ small enough and $N_1$ large enough, we can
ensure that $\lambda_j=N_j\mu(B_j)>\frac{1-\rho_1}{2}$ and thus
$1-\lambda_j<\frac{1+\rho_1}{2}\Rightarrow \nu_1<
(\frac{1+\rho_1}{2})^\ell$ and so, from $(*)$ we have that
\[\nu_1logM<\epsilon_0\qquad\nu_1+\epsilon_0\le
\frac{1}{2}\] hence, from $lemma\;2.3$
\[\psi(K,\lambda_j,\epsilon_0)\le 2^{K\cdot
H((\frac{1+\rho_1}{2})^\ell+\epsilon_0)}\] hence
\[\frac{1}{K}log\cN(\cU_0^{K-1},\rho_2)\le
h_0+\epsilon_0(2+logM)+H((\frac{1+\rho_1}{2})^\ell+\epsilon_0)\Rightarrow\]
\[\limsup_K\frac{1}{K}log\cN(\cU_0^{K-1},\rho_2)\le
h_0+\epsilon_0(2+logM)+H((\frac{1+\rho_1}{2})^\ell+\epsilon_0)\]
letting $\ell\to\infty$ and $\epsilon_0\to 0$ we get
\[\limsup_K\frac{1}{K}log\cN(\cU_0^{K-1},\rho_2)\le h_0\]
as desired.

 \end{proof}

After proving $theorem\;4.2$, we can define, for an ergodic m.t.d.s,
$(X,\cB,\mu,T)$ and a cover $\cU=\{U_1\dots U_M\}$ of $X$, a notion of
measure theoretical entropy in the following way:
\[h_\mu^e(\cU,T)=lim\frac{1}{n}log\cN(\cU_0^{n-1},\epsilon)\quad where
\quad 0<\epsilon<1.\]
Often we omit $T$ and write $h_\mu^e(\cU)$.
\begin{thm}
$h_\mu^e(\cU)=h_\mu^+(\cU)$
\end{thm}
\begin{proof}
As before, if the system is periodic then $h_\mu^e(\cU)=h_\mu^+(\cU)=0$.
We assume, then ,that the system is aperiodic.
For every partition $\alpha\succeq\cU,\;n\in\fN$ and $0<\epsilon <1$, we
have that
$\cN(\cU_0^{n-1},\epsilon)\le \cN(\alpha_0^{n-1},\epsilon)$
and therefore
\[h_\mu^e(\cU)=lim\frac{1}{n}log\cN(\cU_0^{n-1},\epsilon)\le
lim\frac{1}{n}log\cN(\alpha_0^{n-1},\epsilon)=h_\mu(\alpha)\]
\[\Rightarrow h_\mu^e(\cU)\le h_\mu^+(\cU)\]
To prove the other inequality, we shall show that for a given $0<\epsilon
<\frac{1}{4}$ and $n\in\fN$ we have:
\[h_\mu^+(\cU)\le
\frac{1}{n}log\cN(\cU_0^{n-1},\epsilon)+\sqrt{\epsilon}
\cdot logM+H(\sqrt{\epsilon}).\qquad (*)\]
Once we prove $(*)$, we are done, for letting $n\to\infty$ we get
$h_\mu^+(\cU)\le h_\mu^e(\cU)+\sqrt{\epsilon}\cdot
logM+H(\sqrt{\epsilon})$ and
now,
letting $\epsilon\to 0$ we get $h_\mu^+(\cU)\le h_\mu^e(\cU)$ as
desired.\\
Proof of $(*)$: choose $\delta>0$, such that $\epsilon+\delta<\frac{1}{4}$
and find a good base $(\beta,\tilde{B},B)$ for $(\cU,n,\epsilon,\delta)$.
(Now we take $\tilde{B}$ to be a base for a strong Rohlin tower of height
$N$ and error $<\delta$ and not of height $N+1$ as before).
Set $N=\cN(\cU_0^{n-1},\epsilon)$, so $B$ is the union of $N$
elements of $\beta|_{\tilde{B}}$. We index these elements by sequences
$i_0\dots i_{n-1}$, such that  if $B_{i_0\dots i_{n-1}}$ is one, then
$T^j(B_{i_0\dots i_{n-1}})\subset U_{i_j}$, for every $0\le j\le n-1$. We
have that $\mu(X\setminus\bigcup_0^{n-1}T^i(B))\le \epsilon+\delta$.
Let $\hat{\alpha}=\{\hat{A}_1\dots \hat{A}_M\}$ be the partition of
\[E=\bigcup_0^{n-1}T^i(B)\]
defined by
\[\hat{A}_m=\bigcup\{T^j(B_{i_0\dots i_{n-1}})\;|\;j\in
[0,n-1].\;i_j=m\}.\]
Note that $\hat{A}_m\subset U_m$, for every $1\le m\le M$.
Extend $\hat{\alpha}$, to a partition, $\alpha$, of $X$, refining $\cU$,
in some way. Set $\eta^2=\epsilon+\delta$ and define for every $k>n$
$f_k(x)=\frac{1}{k}\sum_o^{k-1}\chi_E(T^jx)$. We have that $0\le f_k\le
1$ and $\int f_k>1-\eta^2$, so if we will denote:
\[G_k=\{x\;|\;f_k(x)>1-\eta\}\]
then,
\[\eta\cdot\mu(G_k^c)\le\int_{G_k^c}1-f_k\le\int 1-f_k\le\eta^2\]
\[\Rightarrow\mu(G_k)\ge 1-\eta.\]
We shall show that we can cover $G_k$, by "few" $(\alpha,[0,k-1])$-names.
Partition $G_k$ according to the values of $0\le i\le k-n$, such that
$T^ix\in B$.
Note that if $x\in G_k$ and $0\le i_1< \dots <i_m\le k-n$, are the times
in which $x$ visits $B$, then the collection $\{[i_j,i_j+n-1]\}_{j=1}^m$
covers all but at most $\eta k+2n$ elements of $[0,k-1]$.
Because each element of this partition defines a collection
of subintervals of $[0,k-1]$, of length $n$, that covers all but at most
$\eta k+2n$, elements of $[0,k-1]$, in a $1-1$ manner, we have that the
number of elements in the partition of $G_k$ is at most
\[\psi(k,n,\eta)=\sum_{j<(\eta+\frac{2n}{k})k}\binom{k}{j}\]
We fix an element $C$ of this partition of $G_k$ and want to estimate the
number of $(\alpha,[0,k-1])$-names, needed to cover it. If $0\le i_1<
\dots <i_m\le k-n$ are the times elements of $C$ visit $B$, then we need
at most $N$, $(\alpha,[i_j,i_j+n-1])$-names, to cover $C$. Because the
size of $[0,k-1]\setminus\bigcup_j[i_j,i_j+n-1]$, is at most $\eta k+2n$,
we need at most $N^{\frac{k}{n}}\cdot M^{\eta k+2n}$
$\quad(\alpha,[0,k-1])$-names, to cover $C$. Finally, we have that we can
cover $G_k$, by no more than:
\[\psi(k,n,\eta)\cdot N^{\frac{k}{n}}\cdot M^{\eta k+2n}\]
$(\alpha,[0,k-1])$-names. Because $\mu(G_k)>1-\eta$, this means that:
\[\frac{1}{k}log\cN(\alpha_0^{k-1},\eta)\le\frac{1}{k}log\psi(k,n,\eta)
+\frac{1}{n}logN+(\eta+\frac{2n}{k})logM.\]
Recall that once $(\eta+\frac{2n}{k})<\frac{1}{2}$, we have
$\psi(k,n,\eta)\le 2^{k\cdot H(\eta+\frac{2n}{k})}$ and so
\[h_\mu(\alpha)=lim\frac{1}{k}log\cN(\alpha_0^{k-1},\eta)\le
\frac{1}{n}log\cN(\cU_0^{n-1},\epsilon)+\eta\cdot logM+H(\eta)\]
so
\[h_\mu^+(\cU)\le
\frac{1}{n}log\cN(\cU_0^{n-1},\epsilon)+\sqrt{\epsilon+\delta}\cdot
logM+H(\sqrt{\epsilon+\delta})\]
Letting $\delta\to 0$ we get
\[h_\mu^+(\cU)\le
\frac{1}{n}log\cN(\cU_0^{n-1},\epsilon)+\sqrt{\epsilon}\cdot
logM+H(\sqrt{\epsilon})\]
as desired.

\end{proof}

\begin{thm}
$h_\mu^+(\cU)=h_\mu^-(\cU)$
\end{thm}
We already know that $h_\mu^+(\cU)\ge h_\mu^-(\cU)$ $(Proposition\;3.6)$,
so we only need to prove the other inequality.
Before we turn to the proof, let us present some terminology and prove a
combinatorial lemma.\\
Let $\Lambda$, be a finite alphabet of $M$ letters, $k,n\in\fN\;k>>n$,
$0<\delta<1$ and $\omega=\omega_0^{k-1}$, a word of length $k$ on
$\Lambda$. (The symbol $a_r^s$ stands for $a_r\dots a_s$).
Denote $\Gamma=\Lambda^n$.
\begin{itemize}
\item An $(n,k,\delta)$-packing is a pair $\cC=(i_0^{m-1},\gamma_0^{m-1})$
where $0\le i_j\le k-n,\;\gamma_j\in\Gamma,\; j=0\dots m-1,\;
i_j+n-1<i_{j+1}$
and $\frac{m\cdot n}{k}>1-\delta$. (We think of an $(n,k,\delta)$-packing
as instructions to "almost" write a word of length $k$, we just fill it
with the $\gamma_j$'s, where $\gamma_j$ starts in the $i_j$ letter and
there will be no more than $\delta k$ letters to add.)
\item An $(n,k,\delta)$-packing for $\omega$, is an
$(n,k,\delta)$-packing, $\cC=(i_0^{m-1},\gamma_0^{m-1})$, such that
$\omega_{i_j}^{i_j+n-1}=\gamma_{j}$.
\item if $\mu_1,\mu_2$ are probability distributions on $\Gamma$ then
\[||\mu_1-\mu_2||=\max_\gamma|\mu_1(\gamma)-\mu_2(\gamma)|.\]
\item An $(n,k,\delta)$-packing, $\cC=(i_0^{m-1},\gamma_0^{m-1})$, induces
a probability distribution on $\Gamma$, denoted by $P_\cC$, by the formula
$P_\cC(\gamma)=\frac{1}{m}\#\{0\le j\le m-1\;|\;\gamma=\gamma_j\}$.
\item If $\mu$ is a probability distribution on $\Gamma$ and $\cC$ is an
$(n,k,\delta)$-packing, then we say that $\cC$ is $(n,k,\delta,\mu)$, if
$||\mu-P_\cC||<\delta$. We say that $\omega$ is $(n,k,\delta,\mu)$, if
there is an $(n,k,\delta)$-packing for $\omega$, which is
$(n,k,\delta,\mu)$.
\end{itemize}
\begin{lem}
If $\mu$ is a probability distribution on $\Gamma$, with "average
entropy"
\[h_0=-\frac{1}{n}\sum_{\gamma\in\Gamma}\mu(\gamma)log\mu(\gamma)\]
then there exists a positive function $\varphi(\delta)$, such that
$\varphi(\delta)\to 0$ as $\delta\to 0$ and such that if
$0<\delta<\frac{1}{2}$, then for any $k>n$, the number of
words $\omega\in\Lambda^k$, which are $(n,k,\delta,\mu)$, is at most
$2^{k(h_0+\varphi(\delta))}$.
 \end{lem}

\begin{proof}
Fix $k>n$. We want to estimate the number of words
$\omega=\omega_0^{k-1}\in\Lambda^k$, that are $(n,k,\delta,\mu)$. For
every such word, $\omega$, we can choose an $(n,k,\delta)$-packing,
$\cC=(i_0^{m-1},\gamma_0^{m-1})$ which is $(n,k,\delta,\mu)$. In this way
we define a map
\[\pi:\{\omega\in\Lambda^k\;|\;\omega\;is\;(n,k,\delta,\mu)\}\rightarrow
\{\cC\;|\;\cC\;is\;an\;(n,k,\delta,\mu)-packing\}\]
If $\cC=(i_0^{m-1},\gamma_0^{m-1})$, is an $(n,k,\delta)$-packing, then
$\frac{n\cdot m}{k}>1-\delta$. This means that $|\pi^{-1}(\cC)|\le
|\Lambda|^{\delta k}=M^{\delta k}$. So we have that
\[\#\{\omega\in\Lambda^k\;|\;\omega\;is\;(n,k,\delta,\mu)\}\le
M^{\delta k}\#\{\cC\;|\;\cC\;is\;an\;(n,k,\delta,\mu)-packing\}.\]
Let us now estimate the number of $(n,k,\delta,\mu)$-packings,
$\cC=(i_0^{m-1},\gamma_0^{m-1})$:\\
The number of sequences, $i_0^{m-1}$
, such that $0\le i_j\le k-n$,
$i_j+n-1<i_{j+1}$ and $\frac{m\cdot n}{k}>1-\delta$ is at most
$\sum_{j<\delta k}\binom{k}{j}$. From $lemma\;2.3$ we know that for
$\delta<\frac{1}{2}$, this sums to
something $\le 2^{H(\delta)k}$. \\
Fix such a sequence $i_0^{m-1}$. Let us now estimate the number of
sequences, $\gamma_0^{m-1}$, such that the $(n,k,\delta)$-packing,
$\cC=(i_0^{m-1},\gamma_0^{m-1})$, is $(n,k,\delta,\mu)$.\\
Denote $\nu=\otimes_1^m\mu$, the product measure on $\Gamma^m$. If
$\gamma_0^{m-1}\in\Gamma^m$, then
\[\nu(\gamma_0^{m-1})=\prod_{\gamma\in\Gamma}\mu(\gamma)^{\#\{0\le j
\le m-1\;|\;\gamma=\gamma_j\}}=
2^{\sum_{\{\gamma|\mu(\gamma)\ne 0\}}\#\{0\le j
\le m-1\;|\;\gamma=\gamma_j\}\cdot log\mu(\gamma)}\]
\[=2^{m\sum_{\{\gamma|\mu(\gamma)\ne 0\}}\frac{1}{m}\#\{0\le j
\le m-1\;|\;\gamma=\gamma_j\}\cdot log\mu(\gamma)}.\]
Now, the function $f:\{(x_\gamma)_{\gamma\in\Gamma}\in\fR^\Gamma\;|\;
\sum x_\gamma=1\}\rightarrow\fR$, defined by
\[f(\vec{x}_\gamma)=
\sum_{\{\gamma|\mu(\gamma)\ne 0\}}x_\gamma\cdot log\mu(\gamma)\]
is continuous and so there is a positive function $\psi(\delta)$, such
that $\psi(\delta)\rightarrow 0$ as $\delta\rightarrow 0$ and if
$\max_\gamma|x_\gamma-\mu(\gamma)|<\delta$, then $|f(\vec{x}_\gamma)-
f(\vec{\mu(\gamma)})|<\psi(\delta)$ (note that $\psi$ depends only on
$n,\;\mu$). So if $\gamma_0^{m-1}\in\Gamma^m$
is such that $\cC=(i_0^{m-1},\gamma_0^{m-1})$, is a
$(n,k,\delta,\mu)$-packing, it follows that
\[\nu(\gamma_0^{m-1})=2^{m\sum_{\{\gamma|\mu(\gamma)\ne
0\}}\frac{1}{m}\#\{0\le j
\le m-1\;|\;\gamma=\gamma_j\}\cdot log\mu(\gamma)}\]
\[\ge
2^{m\big(\sum_{\{\gamma|\mu(\gamma)\ne 0\}}\mu(\gamma)log\mu(\gamma)-
\psi(\delta)\big)}\ge
2^{k(-h_0-\frac{\psi(\delta)}{n})}\]
Where the last inequality follows from the fact that $m<\frac{k}{n}$ and
the definition of $h_0$.
We conclude that an upper bound for the number of such sequences
$\gamma_0^{m-1}$ is $2^{k(h_0+\frac{\psi(\delta)}{n})}$. If we collect
these estimations, we get to the conclusion that for
$0<\delta<\frac{1}{2}$
\[
\#\{\omega\in\Lambda^k\;|\;\omega\;is\;(n,k,\delta,\mu)\}
\le
M^{\delta k}\cdot 2^{H(\delta)k}\cdot 2^{k(h_0+\frac{\psi(\delta)}{n})}
\le
2^{k(h_0+\frac{\psi(\delta)}{n}+H(\delta)+\delta\cdot logM)}\]
so $\varphi(\delta)=\frac{\psi(\delta)}{n}+H(\delta)+\delta\cdot logM$ is
our desired function.

 \end{proof}

\begin{proof}
$(of\;theorem\;4.5)$: We want to show that for an ergodic system
$(X,\cB,\mu,T)$ and a cover $\cU=\{U_1\dots U_M\}$ of $X$, we have
$h_\mu^+(\cU)\le h_\mu^-(\cU)$. As before, if the system is periodic,
then,
from the ergodicity, it must be a cyclic permutation on a finite set of
atoms. Therefore $h_\mu^+(\cU)=h_\mu^-(\cU)=0$. In the aperiodic case we
can use the Strong Rohlin Lemma.\\
Let $\epsilon>0$. We shall show that $h_\mu^+(\cU)\le
h_\mu^-(\cU)+2\epsilon$. From the definition of $h_\mu^-(\cU)$, we
can find $n\in\fN$ and a partition $\beta\succeq\cU_0^{n-1}$, such
that $\frac{1}{n}H_\mu(\beta)\le h_\mu^-(\cU)+\epsilon$. As
$\beta\succeq\cU_0^{n-1}$, we can index the elements of $\beta$,
by sequences $i_0^{n-1}=i_0\dots i_{n-1}$, such that if
$\tilde{B}_{i_0^{n-1}}$, is one, then
$T^j\tilde{B}_{i_0^{n-1}}\subset U_{i_j}\;j=0\dots n-1$. We can
assume that each sequence, $i_0^{n-1}$, corresponds to, at most
one element of $\beta$, for otherwise, we could unite these
elements and get a coarser partition $\beta'$, still refining
$\cU_0^{n-1}$, such that $\frac{1}{n}H_\mu(\beta')\le
\frac{1}{n}H_\mu(\beta)\le h_\mu^-(\cU)+\epsilon$. Set
$\Gamma=\{1\dots M\}^n$. So the elements of $\beta$ are indexed by
$\Gamma$. (if $\gamma\in\Gamma$,does not correspond to an element
of $\beta$, in the above way, we set
$\tilde{B}_\gamma=\emptyset$). In this way, the partition $\beta$,
defines a probability distribution, $\nu$, on $\Gamma$, defined by
$\nu(\gamma)=\mu(\tilde{B}_\gamma)$ and we have that
$h_0=\frac{1}{n}H_\mu(\beta)$, is the "average entropy" (see Lemma
4.6) of
$\nu$.\\
Choose $\delta>0$ (in a manner specified later) and let $F$, be a base for
a strong Rohlin tower (with respect to $\beta$) of height $n$ and
error$\le\delta^2$. Denote the atoms of $\beta|_F$ by
$B_\gamma\;\gamma\in\Gamma$, (where $B_\gamma=\tilde{B}_\gamma\cap F$),
and
define a partition $\tilde{\alpha}=\{\tilde{A}_1\dots \tilde{A}_M\}$ of
$E=\bigcup_0^{n-1}T^jF$, by $\tilde{A}_m=\cup\{T^jB_{i_0^{n-1}}\;|\;
j\in\{0\dots n-1\},\;i_j=m\}$. Note that $\tilde{A}_m\subset U_m$.
Extend $\tilde{\alpha}$, to a partition $\alpha$ of $X$ refining $\cU$, in
some way.
The set of indices of elements of $\alpha$, $\Lambda$ (the alphabet in
which
$\alpha$-names are written) contains $\{1\dots M\}$ and we can always build
$\alpha$, such that $|\Lambda|\le 2M$. We slightly abuse our notation and
denote $\Gamma=\Lambda^n$. In this way, $\nu$ is still a probability
distribution on $\Gamma$.  \\
\emph{Claim}: \emph{If $\delta$, is small enough, then $h_\mu(\alpha)\le
h_0+\epsilon$}.\\
Once we prove this claim, we are done, because then
\[h_\mu^+(\cU)\le
h_\mu(\alpha)\le h_0+\epsilon\le h_\mu^-(\cU)+2\epsilon.\]
\emph{Proof of claim}: For $k>>n$, we look at the function
$f_k(x)=\frac{1}{k}\sum_0^{k-1}\chi_E(T^jx)$. We have that $0\le f_k\le
1$ and $\int f_k>1-\delta^2$. Therefore
\[\delta\cdot\mu(\{x|1-f_k(x)>1-\delta\})\le\int_{\{x|1-f_k(x)>1-\delta\}}1-f_k
\le\int 1-f_k\le\delta^2\]
\[\Rightarrow\mu(\{x|f_k(x)\ge 1-\delta\})\ge
1-\delta.\]
Denote, $G_1^k=\{x|f_k(x)\ge 1-\delta\}$. For $x\in G_1^k$, there are at
most $\delta k$ times $0\le i\le k-1$, such that $T^ix\notin E$. Define
\[G_2^k=\{x|\;|\frac{1}{k}\sum_0^{k-n}\chi_A(T^ix)-\mu(A)|<\delta,\;
A\in\beta|_F\cup\{F\}\}.\]
Let us see what can we say about the $(\alpha,[0,k-1])$-name of an
element,
$x$, of $G_1^k\cap G_2^k$. Fix such an $x$ and denote by
$i_0<\dots<i_{m-1}$, the times between $0$ to $k-n$ in which $x$ visits
$F$. We have that $0\le i_j\le k-n$, $i_j+n-1<i_{j+1}$ (that is because
the
height of the tower is $n$).
Except for at most
$2n$ times ($n$ at the beginning and $n$ at the end), $x$ visits $E$,
exactly in the times $i_j\dots i_j+n-1$, $j=1\dots m-1$. Therefore, we
must have
\[n\cdot m\ge(1-\delta)k-2n\Rightarrow \frac{n\cdot m}{k}\ge
1-(\delta+\frac{2n}{k})\]
Denote the
$(\alpha,
[0,k-1])$-name of $x$ by $\omega=\omega_0^{k-1}$ ($\omega_i\in\Lambda$),
and $\gamma_j=\omega_{i_j}
\dots\omega_{i_j+n-1}\in\Gamma$, $j=0\dots m-1$.
We have that $\cC=(i_0^{m-1},\gamma_0^{m-1})$ is an
$(n,k,\delta+\frac{2n}{k})$-packing for $\omega$. Let us now see, what can
we say about the distribution, $P_\cC$, this packing induces on
$\Gamma$.\\
For $0\le r\le k-n$, we have that $T^rx\in B_\gamma$ if and only if, there
is a $0\le j\le m-1$, such that $r=i_j$ and $\gamma=\gamma_j$. Therefore,
because $x\in G_2^k$
\begin{itemize}
\item $\forall\gamma\in\Gamma\quad|\frac{1}{k}\#\{0\le j\le
m-1|\gamma=\gamma_j\}-\mu(B_\gamma)|<\delta.$
\item $|\frac{m}{k}-\mu(F)|<\delta.$
\end{itemize}
Note that $\mu(F)>\frac{1-\delta}{n}$, so if $\delta$ is sufficiently
small,
we can guarantee that $|\frac{k}{m}-\frac{1}{\mu(F)}|$ would be
arbitrarily small and in turn we can guarantee that for every
$\gamma\in\Gamma$
\[|\frac{k}{m}\cdot\frac{1}{k}\#\{0\le j\le
m-1|\gamma=\gamma_j\}-\frac{\mu(B_\gamma)}{\mu(F)}|=|P_\cC(\gamma)-\nu(\gamma)|
\]
would be arbitrarily small. This is to say that $||P_\cC-\nu||$ is
arbitrarily small. We see that there is a positive function $\psi(\delta)$,
independent of $k$, such that $\psi(\delta)\to 0$ as $\delta\to 0$ and
such that, if $x\in
G_1^k\cap G_2^k$ and $\omega$ is its $(\alpha,[0,k-1])$-name, then
$\omega$ is $(n,k,\psi(\delta)+\frac{2n}{k},\nu)$.\\
Remember the function $\varphi$, from $lemma\;4.6$. There is an $\eta_0>0$,
such that for every $0<\eta<\eta_0\,\;\varphi(\eta)<\epsilon$.
Choose $k$ to
be large enough so that $\frac{2n}{k}<\frac{\eta_0}{2}$ and the
error, $\delta$, of the
tower to be so small, such that $\psi(\delta)<\frac{\eta_0}{2}$, and
conclude, from
$lemme\;4.6$, that the number of $(\alpha,[0,k-1])$-names of elements of
$G_1^k\cap G_2^k$ is at most $2^{k(h_0+\epsilon)}$. From the ergodicity,
we
know that for large enough $k$, $\mu(G_1^k\cap G_2^k)>1-2\delta$, so we
have
\[h_\mu(\alpha)=lim\frac{1}{k}log\cN(\alpha_0^{k-1},2\delta)\le
h_0+\epsilon.\]
as desired.

\end{proof}

Remarks:
\begin{itemize}
\item If $(X,T)$, is totally ergodic, i.e $(X,T^n)$, is ergodic
for every $n\in\fN$, then we can look at expressions like
$h_\mu^e(\cU_0^{n-1},T^n)$. It follows from the definition that
$h_\mu^e(\cU,T)=\frac{1}{n}h_\mu^e(\cU_0^{n-1},T^n)$. This enables
us to prove the last theorem without any hard work done. We know
from $theorem\;4.4$, that $h_\mu^e(\cU,T)=h_\mu^+(\cU,T)$ and
therefore $h_\mu^+(\cU,T)=\frac{1}{n}h_\mu^+(\cU_0^{n-1},T^n)$.
But then, $proposition\;3.6$ (which  is elementary), gives:
$h_\mu^-(\cU,T)=lim\frac{1}{n}h_\mu^+(\cU_0^{n-1},T^n)=h_\mu^+(\cU,T)$
and this gives the desired result. \item The definitions of
$h_\mu^+(\cU),h_\mu^-(\cU)$, were introduced in [R] and discussed
also in [Ye], [HMRY]. There, a proof of their equality was given
only in the case where $(X,T)$, is a t.d.s, and $\cU$ is an open
cover. The proof was based on a reduction to a uniquely ergodic
case and then a use of a variational inequality, proved in [GW].
\item The definition of $h_\mu^e(\cU)$ is new. This definition
helps us to prove directly a slight generalization of the
variational inequality ,proved in [GW] and mentioned above, to the
non-topological case. ($Theorem\;6.1)$.

\item The proofs of theorems 4.2, 4.4, 4.5 and lemma 4.6 are based on ideas
of B.Weiss and E.Glasner

 \end{itemize}

\section{Ergodic decomposition for $h_\mu^+,h_\mu^-$}

\begin{thm}
$(Proposition\;5$ in [HMRY]): Let $\cU=\{U_1\dots U_M\}$, be a cover of
$X$, and $\mu=\int\mu_xd\mu(x)$, the ergodic decomposition of $\mu$ with
respect to $T$. Then
\[h_\mu^+(\cU,T)=\int h_{\mu_x}^+(\cU,T)d\mu(x)\qquad
h_\mu^-(\cU,T)=\int h_{\mu_x}^-(\cU,T)d\mu(x)\]
\end{thm}

\begin{cor}
$h_\mu^+(\cU)=h_\mu^-(\cU)$
\end{cor}
\begin{proof}
It follows immediately from the above and the ergodic case
$(Theorem\;4.5)$
\end{proof}
From now on we will denote the number
$h_\mu^+(\cU,T)=h_\mu^-(\cU,T)(=h_\mu^e(\cU,T)$ in
the ergodic case), simply by $h_\mu(\cU,T)$ or $h_\mu(\cU)$ or $h(\cU)$,
when no ambiguity can occur.


\section{Variational relations}
As always, let $\cU=\{U_1\dots U_M\}$, be a cover of the m.t.d.s
$(X,\cB,\mu,T)$. We can define the "$\emph{combinatorial\;entropy}$" of
$\cU$
as
\[h_{c}(\cU,T)=lim_n\frac{1}{n}log\cN(\cU_0^{n-1})\]
where, $\cN(\cV)$, is the minimum number of elements of $\cV$, needed to
cover the whole space. Note that the sequence $log\cN(\cU_0^{n-1})$, is
sub-additive, hence the limit exists. If $(X,T)$ is a t.d.s and $\cU$ is an open cover then we denote $h_{top}(\cU,T)=h_{c}(\cU,T)$.

The next theorem was proved in [GW] for topological dynamical systems and
measurable covers. We give here a simple proof for the non topological case
that uses the definition of
$h_\mu^e(\cU)$.
\begin{thm}
$h_\mu(\cU)\le h_{c}(\cU)$.
\end{thm}
\begin {proof}
First, if the system is ergodic, then
$h_\mu(\cU)=lim\frac{1}{n}log\cN(\cU_0^{n-1},\frac{1}{2})$
and as $\cN(\cU_0^{n-1},\frac{1}{2})\le\cN(\cU_0^{n-1})$, we have
\[h_\mu(\cU)\le lim\frac{1}{n}log\cN(\cU_0^{n-1})=h_{top}(\cU)\]
as desired.
In the non ergodic case, let $\mu=\int\mu_xd\mu(x)$, be the ergodic
decomposition of $\mu$. By $theorem\;5.1$, $h_\mu(\cU)=\int
h_{\mu_x}(\cU)d\mu(x)$, so from the first part we see that
$h_\mu(\cU)\le h_{c}(\cU)$.
  \end{proof}
Remark: Another simple proof of the above, uses the definition of
$h_\mu^-(\cU)$:
\[H_\mu(\cU_0^{n-1})=\inf_{\alpha\succeq\cU_0^{n-1}}H_\mu(\alpha)\le
\inf_{\alpha\succeq\cU_0^{n-1}}log|\alpha|\le log\cN(\cU_0^{n-1})\]
\[\Rightarrow h_\mu(\cU)=lim\frac{1}{n}H_\mu(\cU_0^{n-1})\le
lim \frac{1}{n}log\cN(\cU_0^{n-1})=h_{c}(\cU).\]

From this stage, until the end of this paper we assume that $(X,T)$, is a
t.d.s. We denote by $\cM_T(X)$, the set of $T$-invariant
probability measures on $X$ and by $\cM_T^e(X)$, the set of ergodic ones.
Also $\cC_X^o$, will denote the set of finite open covers of $X$.

In [BGH], the following theorem was proved:
\begin{thm}
($Theorem\;1$ in [BGH]): If $\cU\in\cC_X^o$, then there exists
$\mu\in\cM_T(X)$, such that $h_\mu(\cU)\ge h_{top}(\cU).$
\end{thm}
In light of $theorem\;6.1$ we have that for every $\cU\in\cC_X^o$, one can
find a measure $\mu\in\cM_T(X)$, such that $h_\mu(\cU)= h_{top}(\cU).$
In fact theorem 7 in [HMRY] now becomes:
\begin{cor}
for every $\cU\in\cC_X^o$, one can
find a measure $\mu\in\cM_T^e(X)$, such that $h_\mu(\cU)= h_{top}(\cU).$
\end{cor}
\begin{proof}
Choose $\mu\in\cM_T(X)$, such that $h_\mu(\cU)= h_{top}(\cU)$, and let
$\mu=\int\mu_xd\mu(x)$, be its ergodic decomposition. We know that
\[h_{top}(\cU)=h_\mu(\cU)=\int h_{\mu_x}(\cU)d\mu(x)\]
and that $h_{\mu_x}(\cU)\le h_{top}(\cU)$. So we must have
$h_{\mu_x}(\cU)= h_{top}(\cU)$ for [$\mu$] a.e $x$.
\end{proof}
We conclude from the above, the classical variational principle:\\
First we state a technical lemma, taken from [Ye].
\begin{lem}
For any $\epsilon>0$, $\mu\in\cM_T(X)$ and $\alpha=\{A_1\dots
A_M\}\in\cP_X$, there exists an open cover $\cU\in\cC_X^o$, such that for
every partition $\beta\succeq\cU$ one has $H_\mu(\alpha|\beta)<\epsilon.$
\end{lem}

\begin{thm}
(The Variational Principle):
\begin{alist}
\item For every $\mu\in\cM_T(X)$, $h_\mu(T)\le h_{top}(T)$.
\item $\sup_{\mu\in\cM_T^e(X)}h_\mu(T)=h_{top}(T)$.
\end{alist}
\end{thm}
\begin{proof}
To prove $(a)$, we first show that for each $\mu\in\cM_T(X)$,
$h_\mu(T)=\sup_{\cU\in\cC_X^o}h_\mu(\cU,T)$.
If this is done, then from $theorem\;6.1$, we get
\[h_\mu(T)\le\sup_{\cU\in\cC_X^o}h_{top}(\cU,T)=h_{top}(T).\]
It follows from the definition, that for any cover $\cU$ of $X$, we have
$h_\mu(\cU,T)\le h_\mu(T)$, so one inequality is clear. For the other
inequality, fix a partition, $\alpha=\{A_1\dots A_M\}$, of $X$ and
$\epsilon>0$. We need to find an open cover, $\cU$, of $X$, such that
$h_\mu(\alpha,T)\le h_\mu(\cU,T)+\epsilon$. By the preceding lemma and
from the fact that for any $\beta\in\cP_X$ one has
$h_\mu(\alpha)\le h_\mu(\beta)+H(\alpha|\beta)$
we have $\cU\in\cC_X^o$, such that
\[h_\mu(\cU,T)=\inf_{\beta\succeq\cU}h_\mu(\beta,T)\ge
\inf_{\beta\succeq\cU}(h_\mu(\alpha,T)-H_\mu(\alpha|\beta))\ge
h_\mu(\alpha,T)-\epsilon.\]
To prove $(b)$, note that from
$(6.3)$ we know that for any $\cU\in\cC_X^o$, we can find
$\mu\in\cM_T^e(X)$, such that $h_\mu(\cU,T)=h_{top}(\cU,T)$. This gives us
\[\sup_{\mu\in\cM_T^e(X)}h_\mu(T)\ge
h_{top}(\cU,T)\Rightarrow\sup_{\mu\in\cM_T^e(X)}h_\mu(T)\ge\sup_{\cU\in\cC_X^o}
h_{top}(\cU,T)=h_{top}(T).\]
Together with $(a)$, we get equality, which is $(b)$.
 \end{proof}


\end{document}